\documentclass[11pt,hyp,]{nyjm}
\usepackage{hyperref}
\hypersetup{nesting=true,debug=true,naturalnames=true}
\usepackage{graphicx,amssymb,upref}

\usepackage{amsmath}
\usepackage{amsthm}
\usepackage{amsfonts}
\usepackage{color}
\usepackage{ytableau}
\ytableausetup{mathmode, boxsize=1.5em}
\usepackage{caption}

\usepackage{latexsym}
\title [Generating functions and maximal tori] {Generating functions and statistics on spaces of maximal tori in classical Lie groups} 

\author{Jason Fulman}  
\address{Department of Mathematics,  University of Southern California\
Los Angeles, CA, United States 90089-2532} 
\email{fulman@usc.edu}  
\thanks{Fulman was partially supported by Simons Foundation grant 400528.} % \thanks entries are to acknowledge grants. 

\author{Rita Jim\'enez Rolland}  
\address{Instituto de Matem\'aticas, Universidad Nacional Aut\'onoma de M\'exico
 Oaxaca de Ju\'arez, Oaxaca,  M\'exico 68000} 
\email{rita@im.unam.mx}  
\thanks{Jim\'enez Rolland  is grateful for the financial support from  PAPIIT-UNAM grant IA100816.}

\author{Jennifer C. H. Wilson}  
\address{Department of Mathematics,  Stanford University\\
   Stanford, CA, United States 94305} 
\email{jchw@stanford.edu}  

\keywords{generating functions, finite groups of Lie type, maximal tori, symmetric group, hyperoctahedral group, signed permutation group, representation stability, homological stability}

\subjclass[2010]{
57T15,  % Homology and cohomology of homogeneous spaces of Lie groups
05A15,  % Exact enumeration problems, generating functions 
20G40,  % Linear algebraic groups over finite fields
20G05   % Representation theory
}

\newcommand{\Z}{\mathbb{Z}}
\newcommand{\Q}{\mathbb{Q}}
\newcommand{\Co}{\mathbb{C}}
\newcommand{\GL}{\mathrm{GL}}

\newcommand{\Sp}{\mathrm{Sp}}
\newcommand{\SO}{\mathrm{SO}}

\newcommand{\GLnq}{GL_n(F_q)}
\newcommand{\Spnq}{Sp_{2n}(F_q)}

\newtheorem{prop}{Proposition}[section]

\newtheorem{lemma}[prop]{Lemma}
\newtheorem{cor}[prop]{Corollary}
\newtheorem{theorem}[prop]{Theorem}

\theoremstyle{definition}
\newtheorem{remark}[prop]{Remark}
\newtheorem{example}[prop]{Example}
\newtheorem{definition}[prop]{Definition}

\usepackage[vcentermath]{youngtab}
\newcommand{\Y}[1]{{\tiny\yng(#1)}}

\begin{document} 
\begin{abstract} In this paper we use generating function methods  to obtain  new  asymptotic results about spaces of $F$-stable maximal tori in $GL_n(\overline{F_q})$, $Sp_{2n}(\overline{F_q})$, and $SO_{2n+1}(\overline{F_q})$. We recover stability results of  Church--Ellenberg--Farb and Jim\'enez Rolland--Wilson for ``polynomial'' statistics on these spaces, and we compute explicit formulas for their stable values. We  derive a double  generating function for the characters of  the cohomology of flag varieties in type B/C, which we use to  obtain analogs in type B/C of results of Chen: we recover ``twisted homological stability'' for the spaces of maximal tori in $Sp_{2n}(\Co)$ and $SO_{2n+1}(\Co)$, and we compute a generating function for their ``stable twisted Betti numbers''. We also give a new proof of a result of Lehrer using symmetric function theory.
\end{abstract}
\maketitle
\tableofcontents

\section{Introduction}

In recent work, Church, Ellenberg, and Farb \cite{CEF} and Jim\'enez Rolland and Wilson \cite{RW} proved asymptotic stability results for certain `polynomial' statistics on the set of  maximal tori in $\GL_n(\overline{F_q})$, $\Sp_{2n}(\overline{F_q})$, and $\SO_{2n+1}(\overline{F_q})$  that are stable under the action of the Frobenius morphism $F$. In this paper we use generating function techniques to give new proofs of these results, and moreover compute explicit formulas for the limiting values of these statistics. We then compute a double generating function for the characters of  the cohomology of flag varieties in type B/C. With these computations we derive analogs in type B/C of results of Chen \cite{Che} in type A: we prove ``twisted homological stability'' for the spaces of maximal tori in $\Sp_{2n}(\Co)$ and  $\SO_{2n+1}(\Co)$, and we give a generating function for the ``stable twisted Betti numbers''. The form of this generating function implies that the stable Betti numbers are quasipolynomial and satisfy linear recurrence relations.

\subsection{Statistics on maximal tori in classical Lie groups}

The paper \cite{CEF} proved results on maximal tori in type $A$, and \cite{RW} on maximal tori in types $B$ and $C$, by using the following formulas. These formulas, first obtained by Lehrer, relate the representation theory of the Lie groups' associated \emph{coinvariant algebras}  with point-counts on the set of  $F$-stable maximal tori.  Each $F$-stable maximal torus $T$ of $\GL_n(\overline{F_q})$ is naturally associated to a conjugacy class of the symmetric group $S_n$; this correspondence is explained in Carter \cite[Chapter 3]{Ca} or Niemeyer--Praeger \cite[Sections 2 and 3]{NP}.  Hence any class function $\chi$ of $S_n$ can be interpreted as a function on maximal tori, with $\chi(T)$ defined  to be  the value of $\chi$ in the corresponding $S_n$--conjugacy class. Similarly, for each $F$-stable maximal torus $T$ of $\Sp_{2n}(\overline{F_q})$  or $\SO_{2n+1}(\overline{F_q})$ there is an associated conjugacy class of the hyperoctahedral group $B_n$. Lehrer proved the following.

\begin{theorem} [\textbf{Statistics on maximal tori} {\cite[Corollary 1.10]{L}}] \label{classf}
\qquad
\begin{itemize}
\item[(i)]  { (\textbf{Type A}).} Let $\chi$ be a class function on  $S_n$ and let $T(n,q)$ denote the set of $F$-stable maximal tori of $\GL_n(\overline{F_q})$.  Then
\begin{equation}\label{lehrers} \frac{1}{ q^{n^2-n}}\sum_{T \in T(n,q)} \chi(T) = \sum_{i= 0}^{{n \choose 2}} q^{-i} \langle \chi,R_n^i \rangle_{S_n}\end{equation} 
where $R_n^i$ is the $i^{th}$ graded piece of the coinvariant algebra $R_n^*$ in type $A$, and $\langle -,-\rangle_{S_n}$ is the standard inner product on $S_n$ class functions.
\item[(ii)]  {(\textbf{Type B/C}).} Let $\chi$ be a class function on $B_n$ and let $T(n,q)$ denote the set of $F$-stable maximal tori of $\SO_{2n+1}(\overline{F_q})$ or $\Sp_{2n}(\overline{F_q})$.  Then
\begin{equation}\label{lehrers2}\frac{1}{q^{2n^2}}\sum_{T \in T(n,q)} \chi(T) = \sum_{i= 0}^{n^2} q^{-i} \langle \chi,R_n^i \rangle_{B_n},\end{equation}
where $R_n^i$ is the $i^{th}$ graded piece of the coinvariant algebra $R_n^*$ in type $B/C$, and $\langle -,-\rangle_{B_n}$ is the standard inner product on $B_n$ class functions.

\end{itemize}
\end{theorem}

It will always be clear from context whether $T(n,q)$ refers to type $A$ or to type $B/C$.
 In Section \ref{GL} we give a new proof of Lehrer's result in type A, Theorem \ref{classf}(i), using symmetric function theory. See Church--Ellenberg--Farb \cite[Theorem 5.3]{CEF} for a proof using the Grothendieck--Lefschetz Theorem.

To obtain their asymptotic stability results on polynomial statistics on the space of maximal tori, Church--Ellenberg--Farb \cite{CEF} and Jim\'enez Rolland--Wilson \cite{RW} study the right-hand side of the formulas in Theorem  \ref{classf}  as $n$ tends to infinity. They use results from the field of \emph{representation stability} to show that for sequences of characters $\chi_n$ that are determined by a \emph{character polynomial} (Definitions \ref{DefnTypeACharPoly} and \ref{DefnTypeBCCharPoly}),  the quantities $\langle \chi_n,R_n^i \rangle_{S_n}$ and  $\langle \chi_n,R_n^i \rangle_{B_n}$  stabilize for $n$ sufficiently large relative to $i$.  They then prove that the stable values are subexponential in $i$, and so conclude that the series on the right-hand side of Equations (\ref{lehrers}) and (\ref{lehrers2}) must converge.

In this paper we study instead the left-hand side of the formulas in Theorem \ref{classf}, using  generating function techniques analogous to the work of Fulman \cite{F} and later Chen \cite{Che}. In Theorem \ref{asym} we give explicit formulas for the limiting values of the left-hand sides of these equations for sequences of characters defined by a character polynomial.

\subsection{Character polynomials in type $A$ and $B/C$.}

Character polynomials for the symmetric groups are implicit in the work of Murnaghan and Specht, and are studied explicitly by Macdonald \cite[I.7.14]{M}. Here we recall their definitions and their analogs in type B/C .

\begin{definition}[\textbf{The functions $X_r$; character polynomials in type $A$}] \label{DefnTypeACharPoly} For a permutation $\sigma \in S_n$ and integer $r \geq 1$,  let $X_r(\sigma)$ denote the number of cycles of size $r$ in the cycle decomposition of $\sigma$. A polynomial in the functions $X_r$ is called a {\it character polynomial} for $S_n$; a character polynomial determines a class function on $S_n$ for each $n\geq 0$. Character polynomials over $\Q$ form a graded ring with generator $X_r$ in degree $r$. \end{definition}

\begin{definition}[\textbf{The $S_n$ character polynomials ${ X \choose \lambda}$}] Given a partition $\lambda$, let $|\lambda|$ be the size of $\lambda$, and let $n_r(\lambda)$ denote the number of parts of $\lambda$ of length $r$. In this notation, define a character polynomial  ${X \choose \lambda}$ of degree $|\lambda|$ by
$$ {X \choose \lambda} (\sigma) := \prod_{r =1}^{|\lambda|} {X_r(\sigma) \choose n_r(\lambda)}  \qquad \text{for all } \sigma \in S_n, \text{ for all } n \in \Z_{\geq 0}.    $$
Character polynomials of the form ${X \choose \lambda}$  span the space of character polynomials with rational coefficients.
\end{definition}

Recall that the conjugacy classes of $B_n$  are parameterized by \emph{double partitions} of $n$, that is, ordered pairs of partitions $(\mu,\lambda)$ such that $|\mu|+|\lambda|=n$. We adopt the convention that the parts of $\mu$ encode the lengths of the  positive cycles of a signed permutation, and the parts of $\lambda$ encode the lengths of the negative cycles.

\begin{definition}[\textbf{The functions $X_r$ and $Y_r$; character polynomials in type $B/C$}]  \label{DefnTypeBCCharPoly}
For $ \sigma \in B_n$, let $X_r(\sigma)$ be the number of positive $r$-cycles of $\sigma$, and let $Y_r(\sigma)$ be the number of negative $r$-cycles. A \emph{character polynomial} for $B_n$ is a polynomial in the graded ring $\mathbb{Q}[X_1,Y_1, X_2, Y_2, \ldots]$. Its degree is defined by assigning $\deg X_r =\deg Y_r =r$ for each $r \geq 1$.  \end{definition}
\begin{definition} {\bf (The $B_n$ character polynomials ${X \choose \mu} {Y \choose \lambda}$).}
The space of rational character polynomials in type $B/C$ is spanned by character polynomials of the following form. Given a double partition $(\mu, \lambda)$, define a degree $|\mu|+ |\lambda|$ character polynomial ${X \choose \mu} {Y \choose \lambda}$ by the formula
$$ {X \choose \mu} {Y \choose \lambda} (\sigma) := \prod_{r = 1}^{|\lambda|+|\mu|} {X_r(\sigma) \choose n_r(\mu)} {Y_r(\sigma) \choose n_r(\lambda)} \qquad \text{for all } \sigma \in B_n, n \in \Z_{\geq 0}.  $$
\end{definition}

 %{\color{cyan} Question: Do we want to add more background on algebra \& topology of spaces of max tori? Response from Jason: you can  decide. Response from Rita: Maybe we can leave as it is and just add a few more explanation about the interpretation of characters? see below}

\subsection{Asymptotic results for polynomial statistics on F-stable maximal tori}
 
Given an $F$-stable  maximal torus $T$  in $GL_{n}(\overline{F_q})$ and a character polynomial $P$, the value $P(T)$ is given by evaluating $P$ on the $S_n$--conjugacy class associated to $T$. Specifically, the statistic $X_r(T)$ counts the number of $r$-dimensional irreducible  $F$-stable subtori of $T$; see Church--Ellenberg--Farb \cite[Section 5]{CEF}. 

For an $F$-stable  maximal torus $T$  in $SO_{2n+1}(\overline{F_q})$  or $Sp_{2n}(\overline{F_q})$ we call the associated double partition $(\mu_T,\lambda_T)$ the \emph{type} of $T$. The value of  a character polynomial $P$ in $T$  is given by the value of $P$ on the conjugacy class $(\mu_T,\lambda_T)$. The statistic $X_r(T)$ (respectively, $Y_r(T)$) counts the number of $r$-dimensional $F$-stable subtori that are irreducible over $F_q$  and that split (respectively, do not split) over $F_{q^r}$.  See Jim\'enez Rolland--Wilson \cite[Section 4.1.4]{RW} for a more detailed explanation of this interpretation.  

In types $A$, $B$, and $C$, we say that a character polynomial $P$ defines a \emph{polynomial statistic} on the corresponding sequence of spaces of maximal tori.  In Theorem \ref{asym} we give formulas for the limits of the expected values of polynomial statistics on $F$-stable maximal tori in type $A$ and $B/C$. To state this theorem, define the following notation.

 For a partition $\lambda$ of $n$, let $$z_{\lambda} = \prod_{r =1}^{|\lambda|} n_r(\lambda)! \; r^{n_r(\lambda)} $$ so $z_{\lambda} $ is defined such that $\frac{n!}{z_{\lambda} }$ is the number of permutations in $S_n$ of cycle type $\lambda$.
Analogously, for a partition $\mu$, let $$ v_{\mu} = \prod_{r =1}^{|\mu|} n_r(\mu)! \;(2r)^{n_r(\mu)} $$ so $\frac{2^n n!}{v_{\mu}v_{\lambda}}$ is the number of signed permutations in $B_n$ of signed cycle type $(\mu, \lambda)$.

\begin{theorem} [\textbf{Asymptotics for polynomial statistics on maximal tori}]\label{asym}\
\begin{itemize}
\item[(i)] {\bf (Type A).}  Let $\lambda$  be a fixed partition. Let  $T(n,q)$ denote the set of $F$-stable maximal tori of $GL_n(\overline{F_q})$. Then
\[ \lim_{n \rightarrow \infty} \frac{1}{q^{n^2-n}} \sum_{T \in T(n,q)} {X \choose \lambda} (T)
= \frac{1}{z_{\lambda}} \prod_{r = 1}^{|\lambda|} \left( \frac{q^r}{q^r-1} \right)^{n_r(\lambda)} .\]

\item[(ii)] {\bf (Type B/C).} Fix partitions $\mu$ and $\lambda$.  Let  $T(n,q)$ denote the set of $F$-stable maximal tori of $SO_{2n+1}(\overline{F_q})$ or $Sp_{2n}(\overline{F_q})$. Then
\begin{eqnarray*}
& &  \lim_{n \rightarrow \infty} \frac{1}{q^{2n^2}} \sum_{T \in T(n,q)} {X \choose \mu} {Y \choose \lambda} (T)\\
& =&  \frac{1}{v_{\mu} v_{\lambda}} \prod_{r = 1}^{|\mu|}  \left( \frac{q^r}{q^r-1} \right)^{n_r(\mu)}
\prod_{r=1}^{|\lambda|}  \left( \frac{q^r}{q^r+1} \right)^{n_r(\lambda)}.
\end{eqnarray*}
\end{itemize}
\end{theorem}
Theorem \ref{asym}(i) is proved in Section \ref{GL} and Theorem \ref{asym}(ii) in Section \ref{Sp}.

In particular, Theorem \ref{asym} implies that the left-hand side of both equations converge for any character polynomial and any field $F_q$. Hence Theorem \ref{asym} recovers a primary result of Church, Ellenberg, and Farb in type A, a consequence of \cite[Theorem 5.6]{CEF}. In types B/C, it recovers a primary result of Jim\'enez Rolland and Wilson, \cite[Theorem 4.3]{RW}. Theorem \ref{asym} improves upon both convergence results by providing explicit formulas for the limits. In Examples \ref{EXAA} and  \ref{EXARW} below we use  Theorem \ref{asym}  to recompute the explicit asymptotic identities from \cite[Section 1]{CEF}  and \cite[Table 1]{RW}. 

%{\color{cyan} Question: Would we like to add some illustrative sample computations of stable expected values to this point in the introduction? Response from Jason: probably not necessary, but you can decide}

\subsection{Cohomology of  flag varieties and maximal tori  in type B/C} \label{SectionCoinvariant}

Chen \cite{Che} uses  generating function techniques to count $F$-stable maximal tori of $GL_n(\overline{F_q})$ and uses the formula in Theorem \ref{classf}(i)  to extract topological information about the complex flag manifolds in type $A$, and the space of maximal tori in the complex group $GL_{n}(\Co)$.  In Section \ref{Sp} we  use Theorem \ref{classf}(ii) and follow his approach to study the cohomology of generalized flag manifolds and the spaces of maximal tori of $Sp_{2n}(\Co)$ and $SO_{2n+1}(\Co)$.

In what follows we denote by $R_n^*$ the \emph{complex coinvariant algebra of type B/C}.  Consider the action of $B_n$ on the polynomial ring $\Co[x_1, \ldots, x_n]$ by permuting and negating the $n$ variables $x_i$. The coinvariant algebra $R_n^*$ is defined as the quotient $R_n^*\cong\Co[x_1, \ldots, x_n] / I_n$ by the homogeneous ideal $I_n$ generated by the $B_n$--invariant polynomials with constant term zero.

Borel  \cite{B} proved that $R^*_n$ is isomorphic as a graded $\Co[B_n]$-algebra to the cohomology of the {\it generalized complete complex flag manifolds}  in type $B$  and $C$.
In type $C$,  the flag manifold is the variety of complete flags equal to their symplectic complements. The flag manifold in type $B$ is the variety of complete flags equal to their orthogonal complements. The cohomology groups of these flag manifolds are supported in even cohomological degree, and the isomorphism from $R^*_n$ to the cohomology ring multiplies the grading by two.

 Wilson \cite{W} uses the theory of \emph{FI$_\mathcal{W}$-modules} to prove that, in each cohomological degree,  the characters of the cohomology of these flag varieties  are given by a character polynomial -- independent of $n$ -- for all $n$ sufficiently large.

 \begin{theorem}[\textbf{Existence of character polynomials for $R_n^i$} {\cite[Corollary 6.5]{W}}] \label{CharPol} For each $i\geq 0$, there exists a unique hyperoctahedral character polynomial $Q_i$ such that $ \chi_{R_n^i} = Q_i$ for all $n$ sufficiently large (depending on $i$). Moreover the degree of the polynomial $Q_i$ is at most $i$.
\end{theorem}

Following a suggestion made by Chen, we obtain in Section \ref{MaxCha} a generating function for the $B_n$ characters of ${R_n^i}$.

\begin{theorem}[\textbf{A generating function for the $B_n$ characters of ${R_n^i}$}] \label{useful} Fix an integer $n\geq 1$. For all signed permutations $\sigma \in B_n$,
\[ \sum_{i=0}^{n^2} \chi_{R_n^i}(\sigma) z^i = \frac{(1-z^2)(1-z^4) \cdots (1-z^{2n})}{\prod_{r=1}^n (1-z^r)^{X_r(\sigma)} (1+z^r)^{Y_r(\sigma)}}. \]
\end{theorem}

Using this result we can recover Theorem \ref{CharPol}  and give a generating function for the polynomials $Q_i$. Chen and Specter \cite{CS} compute the  corresponding result in type $A$.

\begin{theorem}[\textbf{Character polynomials for $R^i_n$ in type B/C}] \label{typBan}
Define character polynomials $Q_i$  by the generating function
\begin{align*} \sum_{i =0}^{\infty} Q_i  \, t^i &= \prod_{k =1}^{\infty} \frac{(1-t^{2k})}{(1-t^k)^{X_k} (1+t^k)^{Y_k}}\\%=\prod_{k = 1} (1-t^k)^{1-X_k} (1+t^k)^{1-Y_k}\\
&=\prod_{k=1}^{\infty} \left(\sum_{j=0}^\infty\binom{1-X_k}{j}(-t^k)^j \right) \left(\sum_{j=0}^\infty\binom{1-Y_k}{j}(t^k)^j\right).\end{align*}
Then the character of the $B_n$-representation $R_n^i$,  the $i^{th}$ graded piece of the coinvariant algebra $R_n^*$ in type $B/C$, is given by $$ \chi_{R_n^i} = Q_i \text{\   \   for all \   \  } i \leq 2n+1.$$
\end{theorem}

Equivalently, in fixed degree $i$, the character of the group $R_n^i$ is given by the character polynomial $Q_i$ for all $n \geq \lceil \frac{i-1}{2} \rceil$.  Theorem \ref{typBan} is proved in Section \ref{Sp}.

\begin{remark} Theorem \ref{typBan} significantly improves upon Wilson's result Theorem \ref{CharPol}: not only does  Theorem \ref{typBan} present an explicit generating function for the character polynomials, it also gives a quantitative range  $i\leq 2n+1$ for when the equality $ \chi_{R_n^i} = Q_i$ holds. The bound $ i \leq 2n+1$ is sharp; see Remark \ref{CoinvariantSharp}. Moreover, from the form of the generating function we can infer  that the polynomial $Q_i$ includes the terms $X_i - Y_i$, hence $\deg(Q_i)=i$. Thus Theorem \ref{typBan}  implies that the bound on the degree of the character polynomials $Q_i$ in Theorem \ref{CharPol} is also sharp.
\end{remark}

%{\color{cyan} Quesiton for Jason and Rita: Would we like to write the statements ``stable range is optimal" or ``bound on deg$(Q_d)$ is sharp" either in the body of Theorem \ref{typBan} or in a numbered Corollary? Or should we keep these statements as comments in the text? Response from Jason: I put these as a remark, feel free to change. Response from Rita: I like it as a remark}

Expanding the generating function in Theorem \ref{typBan}, we recover and extend the list of explicit character polynomials given by Wilson \cite[Section 6]{W}.
\begin{align*}
 1&\;  + \; \Big(X_1 - Y_1 \Big) t
  \;  \\
  & + \;  \Big(-1 +\textstyle\frac12 X_1+ \textstyle\frac12 Y_1 + \textstyle\frac12 X_1^2 + \textstyle\frac12 Y_1^2-X_1 Y_1+ X_2-Y_2  \Big) t^2   \\
+  \Big( &- \textstyle\frac23 X_1 + \textstyle\frac23 Y_1 + \textstyle\frac12 X_1^2-\textstyle\frac12 Y_1^2 +\textstyle\frac16 X_1^3 - \textstyle\frac16 Y_1^3 - \textstyle\frac12 X_1^2 Y_1 \\
& +X_1X_2+Y_1Y_2   + \textstyle\frac12 X_1Y_1^2-X_2Y_1-X_1Y_2+X_3-Y_3\Big)t^3
 \\
+\Big(& -1-\textstyle\frac14 X_1 -\textstyle\frac14 Y_1
-\textstyle\frac1{24}  X_1^2 -\textstyle\frac1{24}  Y_1^2   +\textstyle\frac7{12} X_1Y_1 +\textstyle\frac32 Y_2-\textstyle\frac12X_2  +\textstyle\frac14 X_1^3 
\\ &  \Big.
+ \textstyle\frac14 Y_1^3  -\textstyle\frac14 X_1^2Y_1 - \textstyle\frac14 X_1Y_1^2 +\textstyle\frac12 X_1X_2-\textstyle\frac12 X_1Y_2  -\textstyle\frac12Y_1Y_2 +\textstyle\frac12X_2Y_1
\\ & \Big.
+\textstyle\frac1{24}  X_1^4  +\textstyle\frac1{24}  Y_1^4  -\textstyle\frac16 X_1^3Y_1  -\textstyle\frac16 X_1Y_1^3 +\textstyle\frac14 X_1^2Y_1^2  +\textstyle\frac12 X_1^2X_2 -\textstyle\frac12 X_1^2Y_2
  \\ & \Big.
 +\textstyle\frac12 X_2Y_1^2
  -\textstyle\frac12 Y_1^2Y_2  -\textstyle\frac12 Y_1^2Y_2   -X_1X_2Y_1+X_1Y_1Y_2   +\textstyle\frac12 X_2^2 
 +\textstyle\frac12 Y_2^2  
 \\& 
+X_1X_3 -X_1Y_3  -X_3Y_1 +Y_1Y_3  -X_2Y_2  +X_4-Y_4 \Big)t^4 + \cdots
\end{align*} 

%Maple Code:
% series(simplify(expand((1-binomial(1-X1, 1)*t+binomial(1-X1, 2)*(-t)^2+binomial(1-X1, 3)*(-t)^3+binomial(1-X1, 4)*(-t)^4)*(1+binomial(1-Y1, 1)*t+binomial(1-Y1, 2)*t^2+binomial(1-Y1, 3)*t^3+binomial(1-Y1, 4)*t^4)*(1-binomial(1-X2, 1)*t^2+binomial(1-X2, 2)*(-t^2)^2+binomial(1-X2, 3)*(-t^2)^3+binomial(1-X2, 4)*(-t^2)^4)*(1+binomial(1-Y2, 1)*t^2+binomial(1-Y2, 2)*t^4+binomial(1-Y2, 3)*t^6+binomial(1-Y2, 4)*t^8)*(1-binomial(1-X3, 1)*t^3+binomial(1-X3, 2)*(-t^3)^2+binomial(1-X3, 3)*(-t^3)^3+binomial(1-X3, 4)*(-t^3)^4)*(1+binomial(1-Y3, 1)*t^3+binomial(1-Y3, 2)*t^6+binomial(1-Y3, 3)*t^9+binomial(1-Y3, 4)*t^12)*(1-binomial(1-X4, 1)*t^4)*(1+binomial(1-Y4, 1)*t^4))), t = 0)
% Expansion is valid up to t^4
\subsection{Stable twisted Betti numbers}
%Given a finite group $G$ (typically $S_n$ or $B_n$), we write $\langle -,-\rangle_{G}$ to denote the usual inner product on the space of $G$--class functions. By abuse of notation we may write a $G$-representation $V$ in one or both arguments to indicate the character of  $V$.
\begin{definition}[{\bf Twisted Betti numbers in type B/C}]
Let $T(n,\Co)$ denote the space of maximal tori in $\Sp_{2n}(\Co)$ or (equivalently) $\SO_{2n+1}(\Co)$. For a $B_n$-representation   $V_n$, the \emph{twisted Betti numbers of $T(n,\Co)$ in $V_n$} are the values
$$\dim_\Co H^{2i}(T(n,\Co),V_n)=\langle R^i_n,V_n\rangle_{B_n}.$$
For a virtual $B_n$-representation $V_n$, we again call the values $\langle R^i_n,V_n\rangle_{B_n} 
\in \Co$ the twisted Betti numbers of $T(n,\Co)$ in $V_n$. 
\end{definition}

 Theorem \ref{typBan} and Jim\'enez Rolland--Wilson \cite[Proposition 3.1]{RW} imply that  if $P$ is any hyperoctahedral character polynomial, then the inner product $\langle P, R_n^i \rangle_{B_n}$ is independent of $n$ for  $n \geq \deg(P) + i$. We conclude the following theorem.

\begin{cor}[\textbf{Twisted homological stability for complex tori}]\label{TWISTED} Let $P$ be a hyperoctahedral character polynomial and $V^P_n$ the associated sequence of virtual representations of $B_n$. For every $i\geq 0$ and all  $n \geq \deg(P) + i$,
$$ \langle V^P_n, R_n^i \rangle_{B_n} = \langle V^P_{n+1}, R_{n+1}^i \rangle_{B_{n+1}}.$$
For genuine representations $V^P_n$ this implies in particular
$$ \dim_{\Co} \big(H^{2i}(T(n,\Co),V^P_n)\big)=\dim_{\Co} \big(H^{2i}(T(n+1,\Co),V^P_{n+1})\big).$$
\end{cor}

 We remark that the stable range in Corollary \ref{TWISTED} relies on the result  of Theorem \ref{typBan} that the character of $R^i_n$ agrees with its character polynomial  for all $i \leq 2n+1$, and in particular for all $i \leq n$.

If $P$ is the character polynomial associated to a sequence of irreducible $B_n$-representations (as constructed in Wilson \cite[Theorem 4.11]{W2}), Corollary \ref{TWISTED} can be interpreted as the statement that the cohomology groups of the flag varieties are representation stable in the sense of Church--Farb \cite[Definition 1.1]{CF}, a consequence of Wilson \cite[Corollary 6.5]{W}.

Finally,  in Section \ref{Sp} we prove a type $B/C$ analog of  results of  Chen \cite[Theorem 1 (II) and Corollary 2 (II)]{Che}. First  we give a double generating function for the twisted Betti numbers of $T(n,\Co)$. 

\begin{theorem}[\textbf{The twisted Betti numbers $\left\langle {X \choose \mu} {Y \choose \lambda},R_n^i \right\rangle_{B_n}$}] \label{analog1} Fix a double partition $(\mu, \lambda)$.
Let $$\beta_i(n) = \left\langle {X \choose \mu} {Y \choose \lambda},R_n^i \right\rangle_{B_n}.$$
Then
\begin{align*} &  \sum_{n=0}^{\infty} \sum_{i =0}^{n^2} \frac{\beta_i(n) z^i u^n}{(1-z^2)(1-z^4) \cdots (1-z^{2n})} \\
=& \frac{1}{v_{\mu} v_{\lambda}} \prod_{r=1}^{|\mu|} \left( \frac{u^r}{1-z^r} \right)^{n_r(\mu)}
\prod_{r=1}^{|\lambda|} \left( \frac{u^r}{1+z^r} \right)^{n_r(\lambda)}
\prod_{r =1}^{\infty} \frac{1}{1-uz^{2r-2}} .
\end{align*}
\end{theorem}

We obtain a generating function for the stable twisted Betti numbers  of $T(n,\Co)$. From the form of this generating function we can deduce properties of these sequences such as linear recurrence relationships and \emph{quasipolynomiality}.

\begin{cor}[\textbf{Generating functions for stable Betti numbers}]\label{BettiOGF} Given a double partition $(\mu, \lambda)$, let $$ \beta_i = \lim_{n \rightarrow \infty} \left\langle {X \choose \mu} {Y \choose \lambda} ,R_n^i  \right\rangle_{B_n},$$
where $R_n^i$ is  the $i^{th}$ graded piece of the coinvariant algebra $R_n^*$ in type $B/C$. Then \[\sum_{i = 0}^{\infty} \beta_i z^i = \frac{1}{v_{\mu} v_{\lambda}} \prod_{r=1}^{|\mu|} \left( \frac{1}{1-z^r} \right)^{n_r(\mu)} \prod_{r=1}^{|\lambda|} \left( \frac{1}{1+z^r} \right)^{n_r(\lambda)}.\]
\end{cor}
Since character polynomials of the form ${X \choose \mu} {Y \choose \lambda}$ form an additive basis for the space of all hyperoctahedral character polynomials, the result of Corollary \ref{BettiOGF} is sufficient to derive a generating function for the twisted stable Betti numbers associated to any given character polynomial. Because the generating function in Corollary \ref{BettiOGF} is rational with denominator of degree $|\mu|+|\lambda|$, we obtain the following consequence.

\begin{cor}[\textbf{Linear recurrence for stable Betti numbers}]\label{RECURR} Given a hyperoctahedral character polynomial $P$,  let
$$ \beta_i = \lim_{n \rightarrow \infty} \langle P,R_n^i  \rangle_{B_n},$$
where $R_n^i$ is  the $i^{th}$ graded piece of the coinvariant algebra $R_n^*$ in type $B/C$. Then  there exist
integers $d_1,\cdots,d_N$ such that for all $i \geq N$,
\[ \beta_i = d_1 \beta_{i-1} + d_2 \beta_{i-2} + \cdots + d_N \beta_{i-N}.\]
When $P = {X \choose \mu} {Y \choose \lambda} $, then $N=\deg(P)$. In general $N \leq 2(\deg(P))^2$.
\end{cor}

To state the next consequence of  Corollary \ref{BettiOGF}, we recall the definition of a quasipolynomial from Stanley \cite[Section 4.4]{St1}.

\begin{definition}[\textbf{Quasipolynomials; quasiperiods}] \label{DefnQuasipolynomial} A function $p(t)$ is \emph{quasipolynomial of degree $d$} if it can be expressed in the form $$p(t) = c_d(t) t^d + c_{d-1}(t) t^{d-1} + \cdots + c_1(t) t + c_0(t) $$ where the coefficients $c_i(t)$ are periodic functions of $t$ with integer periods, and $c_d(t)$ is not identically zero. Equivalently, the function $p(t)$ is quasipolynomial if for some $M \geq 1$ there are polynomials $p_0(t), p_1(t), \ldots, p_{M-1}(t)$, such that $$p(t) = p_i(t) \qquad \text{for } t \equiv i \pmod{M}.$$ The integer $M$ is called a \emph{quasiperiod} of $p(t)$.
\end{definition}

\begin{cor}[\textbf{Stable Betti numbers are quasipolynomial}] \label{quasipolynomiality}
Given a hyperoctahedral character polynomial $P$,  again let
$$ \beta_i = \lim_{n \rightarrow \infty} \langle P,R_n^i  \rangle_{B_n},$$
where $R_n^i$ is  the $i^{th}$ graded piece of the coinvariant algebra $R_n^*$ in type $B/C$. Then for $i \geq 1$ the Betti numbers $ \beta_i$ are quasipolynomial in $i$, with degree at most $\deg(P)-1$, and have a quasiperiod at most the least common multiple of $\{ 2k \; | \; 1 \leq k \leq deg(P)\}$. If the character polynomial $P$ has constant term zero, then the above statement holds for all $i \geq 0$. 
\end{cor}

Let $\Co^n$ denote the canonical representation of $B_n$ by signed permutation matrices. To illustrate Corollaries  \ref{BettiOGF}, \ref{RECURR}, and \ref{quasipolynomiality}, Example \ref{ExampleSym2} gives the character polynomial for the sequence of $B_n$-representations $\{\mathrm{Sym}^2 \Co^n\}_n$
and their associated stable twisted Betti numbers.

\begin{example} \label{ExampleSym2}({$\mathrm{Sym}^2 \Co^n$}).

Character polynomial:
\begin{align*}
& X_1 + { X_1 \choose 2} +Y_1+ {Y_1 \choose 2} + X_2 -Y_2 -X_1Y_1
\\ = &   {X \choose \Y{1} } + {X \choose \Y{1,1} } +  {Y \choose \Y{1} } + {Y \choose \Y{1,1} } +  {X \choose \Y{2} } - {Y \choose \Y{2} } - {X \choose \Y{1} } { Y \choose \Y{1} }
\end{align*}

Betti numbers:
\begin{align*} & \sum_{i = 0}^{\infty} \beta_i z^i  =   \frac{-z^4+z^2+1}{(1-z^2)^2(1+z^2)} \\
& =  1+2z^2 + 2z^{4} + 3z^{6} + 3z^{8}  + 4 z^{10} + 4z^{12} + \cdots + \left\lfloor \frac{d+3}{2} \right\rfloor z^{2d} + \cdots
\end{align*}

Recurrence: \qquad $\displaystyle \beta_d = \beta_{d-2} +  \beta_{d-4}  - \beta_{d-6}  \qquad$ for $d \geq 7$ \\

Quasipolynomiality: \qquad For $d \geq 0$,
$$\beta_d = \left\{ \begin{array}{cl} \frac{d+4}{4} & d \equiv 0 \pmod{4} \\[5pt]   \frac{d+6}{4} & d \equiv 2 \pmod{4} \\[5pt]   0 & d \equiv 1,3 \pmod{4}  \end{array} \right. $$

\end{example}
Additional examples of stable twisted Betti numbers are given in Section \ref{SectionExamples}, with their associated generating functions, recurrence relations, and quasipolynomials.

\section{Maximal tori in the general linear groups} \label{GL}
 In this section, we let $T(n,q)$ denote the set of $F$-stable maximal tori of $GL_n(\overline{F_q})$. Fulman \cite[Theorem 3.2]{F} proved the following.
\begin{equation} \label{jef}
1 + \sum_{n = 1}^{\infty} \frac{u^n}{|\GLnq |} \sum_{T \in T(n,q)} \prod_{i= 1}^{\infty} x_i^{X_i(T)}
= \prod_{k = 1}^{\infty} \exp \left[ \frac{x_k u^k}{(q^k-1)k}  \right].
\end{equation}
Chen \cite{Che} deduced from \eqref{jef} that for any partition $\lambda$,
{\small 
\begin{equation} \label{cheneq} \sum_{n=0}^{\infty} \frac{u^n}{|\GLnq|} \sum_{T \in T(n,q)} {X \choose \lambda} (T)
= \frac{1}{z_{\lambda}}  \prod_{k=1}^{|\lambda|} \left( \frac{u^k}{q^k-1} \right)^{n_k(\lambda)} 
\prod_{r =1}^{\infty} \frac{1}{(1-u/q^r)}
\end{equation}
}
 Here if $\lambda$ is the empty partition, we take the constant $n=0$ term on the left hand side to be $1$; otherwise, the constant term is $0$.

To analyze these formulas, we use the following elementary and well-known lemma.
We adopt the following notation: given a power series $f(u)$,
we let $[u^n] f(u)$ denote the coefficient of $u^n$.

\begin{lemma}[\textbf{The limit of a series' coefficients}] \label{tay} If the Taylor series of $f(u)$ around $0$ converges at $u=1$, then
\[ \lim_{n \rightarrow \infty} [u^n] \frac{f(u)}{1-u} = f(1) .\]
\end{lemma}

\begin{proof} Write the Taylor expansion $f(u)=\sum_{n=0}^{\infty} a_n u^n$. Then observe that
\[ [u^n] \frac{f(u)}{1-u} = \sum_{i=0}^n a_i. \qedhere \]
\end{proof}

We can now prove Theorem \ref{asym} (i), which gives formulas for the stable values of polynomial statistics on the space of $F$-stable maximal tori in $GL_n(\overline{F_q})$. Recall that the number of such tori is $q^{n^2-n}$.

%\begin{theorem} \label{asymGL} Let $\lambda=(\lambda_1, \cdots, \lambda_l)$ be fixed. Then
%\[ \lim_{n \rightarrow \infty} \frac{1}{q^{n^2-n}} \sum_{T \in T(n,q)} {X \choose \lambda} (T)
%= \frac{1}{z_{\lambda}} \prod_{k \geq 1} \left[ \frac{q^k}{(q^k-1)} \right]^{\lambda_k} .\]
%\end{theorem}

\begin{proof}[Proof of Theorem \ref{asym} (i)] It follows from \eqref{cheneq} that \[ \frac{1}{q^{n^2-n}} \sum_{T \in T(n,q)}
{X \choose \lambda} (T) \] is equal to
\[ \frac{|\GLnq|}{q^{n^2-n}} [u^n] \frac{1}{z_{\lambda}} \prod_{k=1}^{|\lambda|} \left( \frac{u^k}{q^k-1} \right) ^{n_k(\lambda)}
\cdot \prod_{r =1}^{\infty} \frac{1}{1-u/q^r} .\]
Since $\displaystyle [u^n]f(u) = [u^n]\frac{f(uq)}{q^n}$ for any function $f$, this expression is equal to
\begin{eqnarray*}
& & \frac{|\GLnq|}{q^{n^2}} [u^n] \frac{1}{z_{\lambda}} \prod_{k=1}^{|\lambda|} \left( \frac{u^k q^k}{q^k-1} \right) ^{n_k(\lambda)}
\cdot \frac{1}{1-u} \prod_{r = 1}^{\infty} \frac{1}{1-u/q^r} \\
& = & (1-1/q) \cdots (1-1/q^n) [u^n] \frac{1}{z_{\lambda}} \prod_{k=1}^{|\lambda|} \left( \frac{u^k q^k}{q^k-1} \right) ^{n_k(\lambda)}
\cdot \frac{1}{1-u} \prod_{r = 1}^{\infty} \frac{1}{1-u/q^r}.
\end{eqnarray*}

The result now follows by taking the limit as $n$ tends to infinity, and using Lemma \ref{tay}.  To see that Lemma \ref{tay} is applicable, note that the Taylor series of $\prod_{r \geq 1} \frac{1}{1-u/q^r}$ around $0$ converges at $u=1$. Indeed,
this Taylor series is given by part 2 of Lemma \ref{2props}, and  \[ (1-1/q) \cdots (1-1/q^n) \geq 1 - \frac{1}{q} - \frac{1}{q^2} \]
for all $n$, as is well known from Euler's pentagonal number theorem (page 11 of \cite{A}).

\end{proof}

\begin{remark} We remark that there is an alternate approach to proving Theorem \ref{asym}(i). Chen computed a generating function for  the stable twisted Betti numbers for maximal tori in type A \cite[Proof of Corollary (II)]{Che}. Combining this with \cite[Theorem 5.6]{CEF} proves the result.
\end{remark}

%\begin{remark} We remark that there is an alternate approach to proving Theorem \ref{asym}(i). Chen computed a double generating functions for the twisted Betti numbers for maximal tori in type A \cite[Theorem I (II)]{Che} and a generating function for the stable twisted Betti numbers \cite[Proof of Corollary (II)]{Che}. Using convergence results established by Church, Ellenberg, and Farb in the course of proving \cite[Theorem 5.6]{CEF}, it is possible to manipulate Chen's generating functions to prove the result. \end{remark}

Theorem \ref{asym}(i) gives an efficient way to recover some computations of Church--Ellenberg--Farb \cite{CEF}.

\begin{example}(\cite[Theorems 5.9 \& 5.10]{CEF})\label{EXAA}
\begin{flalign*}(a) & \lim_{n \rightarrow \infty} \frac{1}{q^{n^2-n}} \sum_{T \in T(n,q)} X_1(T) = \frac{q}{q-1}.&&\\
  (b) & \lim_{n \rightarrow \infty} \frac{1}{q^{n^2-n}} \sum_{T \in T(n,q)} \left[ {X_1 \choose 2} (T) - X_2 (T) \right] 
   \\ & 
   = \frac{1}{2} \left( \frac{q}{q-1} \right)^2 - \frac{1}{2} \frac{q^2}{(q^2-1)}\\
& = \frac{1}{q(1-1/q)(1-1/q^2)}.&&\\
\end{flalign*}

\end{example}

\subsection{A new proof of Lehrer's theorem}

In what follows we give a new proof of Theorem \ref{classf}(i) using symmetric
function theory. It should be possible to extend these methods to types B/C, but the argument is more involved and we do not pursue it here.  For the remainder of this section, $R_n^i$ denotes the $i$th graded piece of the coinvariant algebra $R_n^*$ in type $A$. All necessary background in symmetric function theory can be found in Macdonald \cite{M} or Stanley \cite{St}. The following two lemmas will be crucial. Lemma \ref{count} is stated in Fulman \cite{F} and is immediate from Springer--Steinberg \cite[Section 2.7]{SS}.

\begin{lemma}[\textbf{Enumerating tori in $GL_n(\overline{F_q})$ of a given type} \cite{SS, F}] \label{count} Fix a partition $\lambda$ of $n$. The number of $F$-stable maximal tori of $GL_n(\overline{F_q})$ of type $\lambda$ is equal to
\[ \frac{|\GLnq|}{z_{\lambda} \prod_{r=1}^{n} (q^r-1)^{n_r(\lambda)}} .\]
Here $n_r(\lambda)$ is the number of parts of $\lambda$ of size $r$, and
$z_{\lambda}=\prod_{r=1}^{|\lambda|} r^{n_r(\lambda)} n_r(\lambda)!$
\end{lemma}

 Recall that a \emph{standard tableau} of \emph{shape $\lambda$} is a bijective labeling of the boxes of the Young diagram for $\lambda$ by the numbers $1, \cdots, n$ with the property that in each row and in each column the labels are increasing. The \emph{descent set} of such a tableau is the set of numbers $i \in \{1,\cdots,n-1 \}$ for which the box labeled $(i+1)$ is in a lower row than the box labeled $i$. The \emph{major index} $\mathrm{maj}(Y)$ of a tableau $Y$ is the sum of the numbers in its descent set. We let $f_{\lambda,i}$
be the number of standard Young tableaux of shape $\lambda$ and major index $i$.

For example, consider the Young diagram for the partition $\lambda = (2,1,1)$ and its three associated standard tableaux
$$ Y_1= \begin{ytableau} 1&2\\3\\4  \end{ytableau} \qquad Y_2= \begin{ytableau} 1&3\\2\\4 \end{ytableau} \qquad Y_3= \begin{ytableau} 1&4\\2\\3  \end{ytableau} $$
Their descent sets are $\{2,3\}$,  $\{1,3\}$, and $\{1,2\}$, respectively. Hence
$$ \mathrm{maj}(Y_1) = 5 \qquad \mathrm{maj}(Y_2) = 4, \qquad \text{and} \qquad \mathrm{maj}(Y_3) = 3.$$  In this example, 
$$ f_{(2,1,1), i} = \left\{ \begin{array}{ll} 1, & i=3,4,5 \\ 0, & \text{otherwise.} \end{array} \right.$$

%\begin{example}
%A Young diagram for $\lambda =(n-1, 1)$  has $n-1$ boxes in the top row, and one box in the second row. There are precisely $n -1$ standard tableau $Y_1,\ldots,Y_{n-1}$ with this shape, where $Y_i$ is the unique standard tableau of this shape with $ i+ 1$ in the second row.  Since $i$ is the only descent in $Y_i$, the major index  $\mathrm{maj}(Y_i)$ of $Y_i$  is $ i$.
%\end{example}

Lemma \ref{reut} is due to Stanley, Lusztig, and Kraskiewicz--Weyman. See Reutenauer \cite[Theorem 8.8]{Re} for a proof.

\begin{lemma}[\textbf{Decomposing the $S_n$-representations $R^i_n$} {\cite[Theorem 8.8]{Re}}]\label{reut}  The multiplicity of the irreducible
representation of $S_n$ indexed by $\lambda$ in $R_n^i$ is equal to
$f_{\lambda,i}$.
\end{lemma}

We now prove Theorem \ref{classf}(i).

\begin{proof} [Proof of Theorem \ref{classf}(i)] It suffices to prove Theorem \ref{classf}(i) in the special case that $\chi$ is the class function which is equal to $p_{\mu}$ on elements of type $\mu$, where $p_{\mu}$ denotes the power sum symmetric function corresponding to a partition $\mu$. Indeed, taking the coefficient of the Schur function $s_{\lambda}$ in $p_{\mu}$ gives the irreducible character value $\chi^{\lambda}_{\mu}$ of the symmetric group, and these are a basis for the space of class functions.

By Lemma \ref{count}, the left hand side of \eqref{lehrers} is equal to

\[ \frac{1}{q^{n^2-n}} \sum_{|\mu|=n} \frac{|\GLnq|}{z_{\mu} \prod_{k=1}^n (q^k-1)^{n_k(\mu)}} p_{\mu} .\]

On the other hand, since \[ p_{\mu} = \sum_{|\lambda|=n} \chi^{\lambda}_{\mu} s_{\lambda}, \]
it follows from Lemma \ref{reut} that the right hand side of \eqref{lehrers} is equal to

\begin{eqnarray*}
\sum_{i=0}^{{n \choose 2}} q^{-i} \left\langle\sum_{|\lambda|=n} \chi^{\lambda} s_{\lambda}, \sum_{|\lambda|=n}
f_{\lambda,i} \chi^{\lambda}\right\rangle
& = & \sum_{i= 0}^{{n \choose 2}}  q^{-i} \sum_{|\lambda|=n} f_{\lambda,i} s_{\lambda} \\
& = & \sum_{|\lambda|=n} s_{\lambda} \sum_{\mathrm{sh}(Y)=\lambda} \frac{1}{q^{\mathrm{maj}(Y)}},
\end{eqnarray*}
where the sum is over standard Young tableaux $Y$ of shape $\lambda$. By Stanley \cite[p. 363]{St} this is equal to
\begin{align*}
&    \sum_{|\lambda|=n} s_{\lambda} s_{\lambda}(1,1/q,1/q^2,\cdots) \prod_{i=1}^n (1-1/q^i) \\
& =  \frac{|\GLnq|}{q^{n^2-n}} \sum_{|\lambda|=n} s_{\lambda} s_{\lambda}(1/q,1/q^2,\cdots).
\end{align*}
By Macdonald \cite[Section 1.4]{M}, this is
\begin{align*}
&  \frac{|\GLnq|}{q^{n^2-n}} \sum_{|\mu|=n} \frac{1}{z_{\mu}} p_{\mu} p_{\mu}(1/q,1/q^2,\cdots) \\
& =  \frac{1}{q^{n^2-n}} \sum_{|\mu|=n} \frac{|\GLnq|}{z_{\mu} \prod_{k=1}^n (q^k-1)^{n_k(\mu)}} p_{\mu}.
\end{align*}

Thus we have shown that the two sides of the formula in Theorem \ref{classf}(i) are equal, completing the proof.
\end{proof}

\section{Maximal tori in symplectic and special orthogonal groups} \label{Sp}

In this section we use generating function techniques to study $F$-stable maximal tori in $Sp_{2n}(\overline{F_q})$, and (equivalently)
in $SO_{2n+1}(\overline{F_q})$. We prove  Theorem \ref{asym}(ii) and use it to recover results from Jim\'enez Rolland--Wilson \cite{RW}. In the second half of the section, we  derive analogs of results from Chen's lovely paper \cite{Che}  for the special orthogonal and symplectic groups, as well as analogs of results in a personal communication from Chen and Specter \cite{CS}.

The next two results are symplectic versions of results from Fulman \cite{F}.

\begin{lemma}[\textbf{Enumerating tori in $Sp_{2n}(\overline{F_q})$ of a given type}] \label{typeSp} Let $(\mu,\lambda)$ be an ordered pair of partitions satisfying $|\mu|+|\lambda|=n$.
 Let $n_r(\mu)$ denote the number of parts of $\mu$ of size $r$, and let $n_r(\lambda)$ denote the number of parts of $\lambda$ of size $r$. Then the number of $F$-stable maximal tori of $Sp_{2n}(\overline{F_q})$ of type $(\mu,\lambda)$ is equal to
\[ \frac{|\Spnq|}{\prod_{r=1}^n n_r(\mu)! \, n_r(\lambda)! \, (2r)^{n_r(\mu)+n_r(\lambda)}
\prod_{r=1}^n (q^r-1)^{n_r(\mu)} (q^r+1)^{n_r(\lambda)}}. \]
\end{lemma}

\begin{proof} This is immediate from Springer--Steinberg \cite[Section 2.7]{SS}, together with the fact that the
centralizer size of an element of $B_n$ of type $(\mu,\lambda)$ is equal to
\[ \prod_{r=1}^{n} n_r(\mu)! \,  n_r(\lambda)! \, (2r)^{n_r(\mu)+n_r(\lambda)}. \qedhere \] \end{proof}

 In this section, $T(n,q)$ denotes the set of $F$-stable maximal tori $T$ of $Sp_{2n}(\overline{F_q})$ (or $SO_{2n+1}(\overline{F_q})$).

\begin{theorem}[\textbf{A generating function for statistics on $T(n,q)$}]  \label{Spgen}
\begin{align*}
&   1 + \sum_{n = 1}^{\infty} \frac{u^n}{|\Spnq|} \sum_{T \in T(n,q)} \prod_{r=1}^{n} x_r^{X_r(T)} y_r^{Y_r(T)} \\
& = \prod_{k =1}^{\infty} \exp \left[ \frac{x_k u^k}{(q^k-1)2k} + \frac{y_k u^k}{(q^k+1)2k} \right].
 \end{align*}
\end{theorem}

\begin{proof} Lemma \ref{typeSp} implies that the left-hand side of the theorem is equal to
\[ 1 + \sum_{n =1}^{\infty} \sum_{ (\mu,\lambda)\atop |\mu|+|\lambda|= n}
\frac{u^n \prod_{r=1}^n  x_r^{n_r(\mu)} y_r^{n_r(\lambda)}} {\prod_{r} n_r(\mu)! \, n_r(\lambda)! \, (2r)^{n_r(\mu)+n_r(\lambda)}
(q^r-1)^{n_r(\mu)} (q^r+1)^{n_r(\lambda)}} \] By the Taylor expansion of the exponential function,
this is equal to \[ \prod_{k = 1}^{\infty} \exp \left[ \frac{x_k u^k}{(q^k-1)2k} + \frac{y_k u^k}{(q^k+1)2k} \right]. \qedhere \]
\end{proof}

The following lemma will be helpful in manipulating the generating function of Theorem \ref{Spgen}.

\begin{lemma}[\textbf{Two generating function identities}] \label{2props} \quad
\begin{enumerate}
\item[(i)] $ \displaystyle \prod_{i =1}^{\infty} \exp \left[ \frac{1}{(q^i-1)} \frac{u^i}{i} \right] = \prod_{r= 1}^{\infty} \frac{1}{1-u/q^r} .$
\item[(ii)]  $\displaystyle \prod_{r= 1}^{\infty} \frac{1}{(1-u/q^r)} = 1 + \sum_{n=1}^{\infty} \frac{u^n}{q^n (1-1/q) \cdots (1-1/q^n)} .$
\end{enumerate}
\end{lemma}

\begin{proof} The first assertion is proved in the proof of Fulman \cite[Theorem 3.4]{F}.
The second assertion is classical and goes back to Euler; see Andrews \cite[Corollary 2.2]{A}. \end{proof}

As a corollary, we recover Steinberg's enumeration of the number of $F$-stable
maximal tori of $Sp_{2n}(\overline{F_q})$ \cite[Corollary 14.16]{Steinberg}. Recall that given a power series $f(u)$, we let $[u^n] f(u)$ denote the coefficient of $u^n$.

\begin{cor}[\textbf{Enumerating tori in $Sp_{2n}(\overline{F_q})$} {\cite[Corollary 14.16]{Steinberg}}] \label{countSptori} The number of $F$-stable maximal tori of $Sp_{2n}(\overline{F_q})$ is $q^{2n^2}$.
\end{cor}

\begin{proof} By setting all variables $x_r=y_r=1$ in Theorem \ref{Spgen}, we find that the number of maximal tori is equal to
\begin{align*}
&  |\Spnq| \, [u^n] \prod_{r=1}^{\infty} \exp \left[ \frac{u^r}{(q^r-1)2r} + \frac{u^r}{(q^r+1)2r} \right] \\
& = |\Spnq| \, [u^n] \prod_{r=1}^{\infty} \exp \left[ \frac{u^r q^r}{r (q^{2r}-1)} \right].
\end{align*}
Since $\displaystyle [u^n]f(u) = q^n [u^n] f(u/q)$ for any function $f$, this coefficient is equal to
\begin{eqnarray*}
& & |\Spnq| \, q^n [u^n] \exp \left[ \sum_{r=1}^{\infty} \frac{u^r}{r(q^{2r}-1)} \right].
\end{eqnarray*}
By Lemma \ref{2props} (i), this is
\begin{eqnarray*}
&  & |\Spnq| \,  q^n [u^n] \prod_{r = 1}^{\infty} \frac{1}{(1-u/q^{2r})}
\end{eqnarray*}
and by Lemma \ref{2props} (ii), this expression equals
\begin{eqnarray*}
&  & |\Spnq| \,  q^n \frac{1}{q^{2n} (1-1/q^2) \cdots (1-1/q^{2n})}.
\end{eqnarray*}
Finally, we use the identity $\displaystyle |\Spnq| = q^{n^2} \prod_{i=1}^n (q^{2i}-1) $ to conclude that the number of  $F$-stable maximal tori is $q^{2n^2}$ as claimed.
\end{proof}

Table 1 of the paper Jim\'enez Rolland--Wilson \cite{RW} gives explicit expressions for the average values of the following statistics on $T(n,q)$:  $$ X_1 , \qquad  X_1 + Y_1, \qquad  {X_1+ Y_1 \choose 2} - (X_2 + Y_2), \qquad X_2 - Y_2 $$

These average values can be derived more efficiently from the generating function given in Theorem  \ref{basic}, and their limiting values are special cases of Theorem \ref{asym} (ii). Theorem \ref{basic}, which we now prove, is an analog of Chen's Equation \eqref{cheneq} from Section \ref{GL}.

\begin{theorem} [\textbf{Average values of the statistic ${X \choose \mu} {Y \choose \lambda}$ on $T(n,q)$}] \label{basic} Fix a double partition $(\mu, \lambda)$. Then
\[ \sum_{n= 0}^{\infty} \frac{u^n}{|\Spnq|} \sum_{T \in T(n,q)} {X \choose \mu} {Y \choose \lambda} (T) \]
is equal to
\[ \frac{1}{v_{\mu} v_{\lambda}} \prod_{r=1}^{|\mu|} \left( \frac{u^r}{q^r-1} \right)^{n_r(\mu)}
\prod_{r=1}^{|\lambda|} \left( \frac{u^r}{q^r+1} \right)^{n_r(\lambda)}
\prod_{k =1}^{\infty} \frac{1}{1-u/q^{2k-1}}. \]
\end{theorem} 
\noindent If both $\mu$ and $\lambda$ are the empty partitions, we define the constant $n=0$ term to be $1$; otherwise, the constant term is $0$.

\begin{proof} We take the generating function of Theorem \ref{Spgen}, and for each $r \geq 1$ we differentiate $n_r(\mu)$ many times with respect to the variable $x_r$,
and $n_r(\lambda)$ many times with respect to $y_r$. We then set all variables $x_r=y_r=1$. The result is that the quantity \[ \sum_{T \in T(n,q)} {X \choose \mu} {Y \choose \lambda} (T) \]
is equal to
\begin{eqnarray*}
 && \frac{|\Spnq|}{v_{\mu} v_{\lambda}} [u^n] \prod_{r=1}^{|\mu|} \left( \frac{u^r}{q^r-1} \right)^{n_r(\mu)}\prod_{r=1}^{|\lambda|} \left( \frac{u^r}{q^r+1} \right)^{n_r(\lambda)}
  \\ &&  \cdot
   \prod_{k=1}^{\infty} \exp \left[ \frac{u^k}{(q^k-1)2k} + \frac{u^k}{(q^k+1)2k} \right].
\end{eqnarray*}
By Lemma \ref{2props} (i), this is
\[ \frac{|\Spnq|}{v_{\mu} v_{\lambda}} [u^n] \prod_{r=1}^{|\mu|} \left( \frac{u^r}{q^r-1} \right)^{n_r(\mu)}
\prod_{r=1}^{|\lambda|} \left( \frac{u^r}{q^r+1} \right)^{n_r(\lambda)}
\prod_{k=1}^{\infty} \frac{1}{1-u/q^{2k-1}}\] as claimed.
\end{proof}

%\begin{theorem} \label{doubleas} Fix partitions $\mu$ and $\lambda$. Then
%\begin{eqnarray*}
%& &  \lim_{n \rightarrow \infty} \frac{1}{q^{2n^2}} \sum_{T \in T(n,q)} {X \choose \mu} {Y \choose \lambda} (T) \\
%& = & \frac{1}{v_{\mu} v_{\lambda}} \prod_i \left( \frac{q^i}{q^i-1} \right)^{n_i(\mu)}
%\cdot \prod_i \left( \frac{q^i}{q^i+1} \right)^{n_i(\lambda)}
%\end{eqnarray*}
%\end{theorem}

With the result of Theorem \ref{basic}, we can now prove Theorem \ref{asym} (ii), which gives explicit formulas for the stable mean values of the polynomial statistics ${X \choose \mu} {Y \choose \lambda}$ on $T(n,q)$, in the limit as $n$ tends to infinity.

\begin{proof}[Proof of Theorem \ref{asym} (ii)] By Theorem \ref{basic},
\[ \frac{1}{q^{2n^2}} \sum_{T \in T(n,q)} {X \choose \mu} {Y \choose \lambda} (T) \] is equal to
\[ \frac{|\Spnq|}{q^{2n^2} \, v_{\mu} \, v_{\lambda}} [u^n] \prod_{r=1}^{|\mu|} \left( \frac{u^r}{q^r-1} \right)^{n_r(\mu)}
\prod_{r=1}^{|\lambda|} \left( \frac{u^r}{q^r+1} \right)^{n_r(\lambda)}
\prod_{k=1}^{\infty} \frac{1}{1-u/q^{2k-1}}. \]

Since $\displaystyle [u^n]f(u)=[u^n]\frac{f(uq)}{q^n}$ for any function $f$, this can be rewritten as
\[ \frac{|\Spnq|}{q^{2n^2+n} \; v_{\mu} \, v_{\lambda}} [u^n]\prod_{r=1}^{|\mu|}\left( \frac{u^r q^r}{q^r-1} \right)^{n_r(\mu)}
\prod_{r=1}^{|\lambda|} \left( \frac{u^r q^r}{q^r+1} \right)^{n_r(\lambda)}
\prod_{k =1}^{\infty} \frac{1}{1-u/q^{2k-2}}, \] which is equal to
\begin{eqnarray*}
& & (1-1/q^2)(1-1/q^4) \cdots (1-1/q^{2n}) \\
& & \cdot \frac{1}{v_{\mu} \, v_{\lambda}}[u^n] \prod_{r=1}^{|\mu|}\left( \frac{u^r q^r}{q^r-1} \right)^{n_r(\mu)}
\prod_{r=1}^{|\lambda|} \left( \frac{u^r q^r}{q^r+1} \right)^{n_r(\lambda)}
\frac{1}{1-u} \prod_{k=1}^{\infty} \frac{1}{1-u/q^{2k}}.
\end{eqnarray*}

Now applying Lemma \ref{tay} proves the theorem.
\end{proof}

With  Theorem \ref{asym}(ii) we recover the following four results from Table 1 of Jim\'enez Rolland--Wilson \cite{RW}.\\

\begin{example}[{\cite[Table 1]{RW}}]\label{EXARW}
{\small
\begin{flalign*}(a) & \lim_{n \rightarrow \infty} \frac{1}{q^{2n^2}} \sum_{T \in T(n,q)}  X_1(T)  = \frac{q}{2(q-1)}.\\
  (b) & \lim_{n \rightarrow \infty} \frac{1}{q^{2n^2}} \sum_{T \in T(n,q)} \left[ X_1(T) + Y_1(T) \right] =\frac{q}{2(q-1)} + \frac{q}{2(q+1)} = \frac{q^2}{q^2-1}.
  \end{flalign*}
\begin{flalign*}
(c) &  \lim_{n \rightarrow \infty} \frac{1}{q^{2n^2}} \sum_{T \in T(n,q)}  \left[ {X_1(T) + Y_1(T) \choose 2} - (X_2(T) + Y_2(T)) \right] \\
  = &\lim_{n \rightarrow \infty} \frac{1}{q^{2n^2}} \sum_{T \in T(n,q)}  \left[ {X_1(T) \choose 2} + {Y_1(T) \choose 2} + X_1(T) Y_1(T) - (X_2(T) + Y_2(T)) \right] \\
=& \frac{1}{2} \left[ \frac{q}{2(q-1)} \right]^2 + \frac{1}{2} \left[ \frac{q}{2(q+1)} \right]^2
+ \frac{q}{2(q-1)} \frac{q}{2(q+1)} - \frac{q^2}{4(q^2-1)} - \frac{q^2}{4(q^2+1)} \\
 = &  \frac{q^4}{(q^2-1)(q^4-1)}.
 \end{flalign*}
\begin{flalign*}
(d) & \lim_{n \rightarrow \infty} \frac{1}{q^{2n^2}} \sum_{T \in T(n,q)} \left[ X_2(T) - Y_2(T) \right] =\frac{q^2}{4(q^2-1)} - \frac{q^2}{4(q^2+1)} = \frac{q^2}{2(q^4-1)}. &
\end{flalign*}
}

\end{example}

\subsection{Maximal tori in $Sp_{2n}(\Co)$ and $SO_{2n+1}(\Co)$}\label{MaxCha}

Next we  follow Chen and Specter's approach \cite{Che, CS} and use our statistical computations on maximal tori in $Sp_{2n}(\overline{F_q})$ and $SO_{2n+1}(\overline{F_q})$ to obtain information about cohomology of the spaces of maximal tori $T(n,\Co)$ in the complex algebraic groups $Sp_{2n}(\Co)$ and $SO_{2n+1}(\Co)$. The bridge between the combinatorial and the topological data is given by  Lehrer's Theorem \ref{classf} (ii).  In the remainder of this section  $R_n^i$ always denotes the $i^{th}$ graded piece of the complex coinvariant algebra $R_n^*$ in type $B/C$. We first prove Theorem \ref{useful}.% which will be helpful in the proof of Theorem \ref{typBan}.

\begin{proof}[Proof of Theorem \ref{useful}] Fix $n$ and $\sigma \in B_n$. In Theorem \ref{classf} (ii) let $\chi$ be the class function on $B_n$ which is $1$ on elements
of type $\sigma$ and $0$ else. Since the centralizer size of $\sigma$ is equal to
\[ \prod_{r=1}^n X_r(\sigma)! Y_r(\sigma)! (2r)^{X_r(\sigma)+Y_r(\sigma)} \] it follows that the number of maximal tori of type $\sigma$ is equal to
\[ q^{2n^2} \sum_{i=0}^{n^2} \chi_{R_n^i}(\sigma) q^{-i} \frac{1}{\prod_{r=1}^n X_r(\sigma)! Y_r(\sigma)! (2r)^{X_r(\sigma)+Y_r(\sigma)}}.\]

On the other hand, Lemma \ref{typeSp} indicates that the number of maximal tori of type $\sigma$ is equal to
\[ \frac{|\Spnq |}{\prod_{r=1}^n X_r(\sigma)! Y_r(\sigma)! (2r)^{X_r(\sigma)+Y_r(\sigma)} \prod_{r=1}^n (q^r-1)^{X_r(\sigma)} (q^r+1)^{Y_r(\sigma)}} .\]
Thus
\[ q^{2n^2} \sum_{i=0}^{n^2} \chi_{R_n^i}(\sigma) q^{-i} = \frac{|\Spnq |}{\prod_{r=1}^n (q^r-1)^{X_r(\sigma)} (q^r+1)^{Y_r(\sigma)}} .\]

Since $|\Spnq  |=q^{2n^2+n} (1-1/q^2)(1-1/q^4) \cdots (1-1/q^{2n})$, it follows that
\[ \sum_{i=0}^{n^2} \chi_{R_n^i}(\sigma) q^{-i} = \frac{(1-1/q^2)(1-1/q^4) \cdots (1-1/q^{2n})}
{\prod_{r=1}^n (1-1/q^r)^{X_r(\sigma)} (1+1/q^r)^{Y_r(\sigma)}} .\] Since this equality holds for any prime power $q$, we can set $z=1/q$ to 
complete the proof.
\end{proof}

Now we prove Theorem \ref{typBan}, giving a generating function for the character polynomials associated to the cohomology groups of the generalized flag varieties in type B/C.

\begin{proof} [Proof of Theorem \ref{typBan}] Fix a positive integer $n$ and signed permutation $\sigma \in B_n$.  Let $Q_i$ be the character polynomial defined by the generating function
\[  \sum_{i=0}^{\infty} Q_i(\sigma) t^i = \prod_{k=1}^{\infty} \frac{(1-t^{2k})}{(1-t^k)^{X_k(\sigma)} (1+t^k)^{Y_k(\sigma)}} \]
Our goal is to show that the terms in the series
\[ \sum_{i=0}^{n^2} \chi_{R_n^i}(\sigma) t^i - \sum_{i=0}^{\infty} Q_i(\sigma) t^i \]
 vanish for $i \leq 2n+1$. By Theorem \ref{useful} this difference is equal to
\begin{align*}
& \prod_{k=1}^n \frac{(1-t^{2k})}{(1-t^k)^{X_k(\sigma)} (1+t^k)^{Y_k(\sigma)}} -
\prod_{k=1}^{\infty} \frac{(1-t^{2k})}{(1-t^k)^{X_k(\sigma)} (1+t^k)^{Y_k(\sigma)}} \\
& =  \prod_{k=1}^n \frac{(1-t^{2k})}{(1-t^k)^{X_k(\sigma)} (1+t^k)^{Y_k(\sigma)}}
\left[ 1 - \prod_{k > n} \frac{(1-t^{2k})}{(1-t^k)^{X_k(\sigma)} (1+t^k)^{Y_k(\sigma)}} \right] \\
\end{align*}
Since $\sigma \in B_n$, the class functions $X_k(\sigma)$ and $Y_k(\sigma)$ vanish when $k>n$, and so this series equals
\begin{eqnarray*}
\prod_{k=1}^n \frac{(1-t^{2k})}{(1-t^k)^{X_k(\sigma)} (1+t^k)^{Y_k(\sigma)}} \left[ 1 - \prod_{k > n} (1-t^{2k}) \right].
\end{eqnarray*} The smallest power of $t$ in this series is $t^{2n+2}$, which completes the proof.
\end{proof}

\begin{remark}[\textbf{The stable range in Theorem \ref{typBan} is sharp}] \label{CoinvariantSharp} The bound given in Theorem \ref{typBan} is optimal: since for all $\sigma \in B_n$ the power series   \begin{eqnarray*}
\prod_{k=1}^n \frac{(1-t^{2k})}{(1-t^k)^{X_k(\sigma)} (1+t^k)^{Y_k(\sigma)}} \left[ 1 - \prod_{k > n} (1-t^{2k}) \right]
\end{eqnarray*}  appearing in the proof includes the term $t^{2n+2}$, the equality  $\chi_{R_n^i} = Q_i$ holds only for $i \leq 2n+1$. We can see this concretely in small degrees: the spaces $R^2_0$ and $R^4_1$ are both zero, but by inspection the character polynomials $Q_2$ and $Q_4$ for the sequences $\{R^2_n\}_n$ and $\{R^4_n\}_n$  (given explicitly in Section \ref{SectionCoinvariant}) do not vanish identically on the groups $B_0$ and $B_1$, respectively. %Hence the bound is sharp  for $n=0$ and $n=1$, which implies that $d \geq 2n+1$ is the optimal linear range.
\end{remark}

Next we prove  Theorem \ref{analog1}, which  gives a double generating function for the twisted Betti numbers of $T(n,\Co)$.

\begin{proof}[Proof of Theorem \ref{analog1}] For $q$ a prime power,
\begin{eqnarray*}
& & \sum_{n =0}^{\infty} \sum_{i =0}^{n^2} \frac{\beta_i(n)}{(1-1/q^2) \cdots (1-1/q^{2n})} q^{-i} u^n \\
& = & \sum_{n=0}^{\infty} \frac{1}{|\Spnq|} \left[ q^{2n^2} \sum_{i = 0}^{n^2} \beta_i(n) q^{-i} \right] (uq)^n
\end{eqnarray*} By Theorem \ref{classf} (ii), this is
\begin{eqnarray*}
 \sum_{n = 0}^{\infty} \frac{1}{|\Spnq|} \left[ \sum_{T \in T(n,q)} {X \choose \mu} {Y \choose \lambda} (T) \right] (uq)^n
\end{eqnarray*} which (by evaluating the series in Theorem \ref{basic} at $uq$) equals
\begin{eqnarray*}
\frac{1}{v_{\mu} v_{\lambda}} \prod_{r=1}^{|\mu|} \left( \frac{u^r q^r}{q^r-1} \right)^{n_r(\mu)}
\prod_{r = 1}^{|\lambda|} \left( \frac{u^r q^r}{q^r+1} \right)^{n_r(\lambda)} \prod_{r =1}^{\infty} \frac{1}{1-u/q^{2r-2}}.
\end{eqnarray*}
Since the equality holds for all prime powers $q$, the equality also holds when $q^{-1}$ is replaced
by a formal variable $z$.
\end{proof}

We now prove Corollary \ref{BettiOGF}, a generating function for the stable twisted Betti numbers associated to the character polynomial ${X \choose \mu}{Y \choose \lambda}$. We are grateful to Weiyan Chen for help with filling in the details of this proof. 

\begin{proof}[Proof of Corollary \ref{BettiOGF}] 
Fix a double partition $(\mu, \lambda)$, and let
$$ \beta_i(n) = \left\langle {X \choose \mu}{Y \choose \lambda}, R_n^i \right\rangle_{B_n} .$$

For each $n \geq 0$, define the power series $f_n(z)$ by 
$$  f_n(z) :=\sum_{i = 0}^{n^2} \frac{\beta_i(n) z^i}{(1-z^2)(1-z^4) \cdots (1-z^{2n})}.$$
By Theorem \ref{analog1}, we have equality
\begin{align*}
 \sum_{n=0}^{\infty} f_n(z) u^n  
 &=\sum_{n = 0}^{\infty} \sum_{i = 0}^{n^2} \frac{\beta_i(n) z^i u^n}{(1-z^2)(1-z^4) \cdots (1-z^{2n})} \\ 
 &= \frac{1}{v_{\mu}\, v_{\lambda}} \prod_{r=1}^{|\mu|} \left( \frac{u^r}{1-z^r} \right)^{n_r(\mu)} 
\prod_{r=1}^{|\lambda|} \left( \frac{u^r}{1+z^r} \right)^{n_r(\lambda)} \prod_{r=1}^{\infty} \frac{1}{1-uz^{2r-2}}. 
\end{align*}

Let $f(z)$ denote the pointwise limit of the sequence $\{f_n(z)\}$ on the open unit disk $\{z \in \mathbb{C} \; : \; |z|<1\}$. Then for fixed $z$ with $|z|<1$, 
\begin{align*}
 f(z) &=  \lim_{n \to \infty} [u^n]  \sum_{n=0}^{\infty} f_n(z) u^n \\ 
  & =   \lim_{n \rightarrow \infty} [u^n] \frac{1}{v_{\mu}\, v_{\lambda}} \prod_{r=1}^{|\mu|} \left( \frac{u^r}{1-z^r} \right)^{n_r(\mu)}
\prod_{r=1}^{|\lambda|} \left( \frac{u^r}{1+z^r} \right)^{n_r(\lambda)} \prod_{r=1}^{\infty} \frac{1}{1-uz^{2r-2}}.
\end{align*}
By Lemma \ref{tay} this pointwise limit is
$$ f(z) = \frac{1}{v_{\mu} \, v_{\lambda}} \prod_{r=1}^{|\mu|} \left( \frac{1}{1-z^r} \right)^{n_r(\mu)} \prod_{r=1}^{|\lambda|} \left( \frac{1}{1+z^r} \right)^{n_r(\lambda)}
\prod_{r=1}^{\infty} \frac{1}{1-z^{2r}}.$$ 

Next, recall that by Wilson \cite[Corollary 6.5]{W} and Jim\'enez Rolland--Wilson \cite[Proposition 3.1]{RW}, for each fixed $i$ the sequence $\beta_i(n)$ is eventually constant, with stable value $\beta_i := \lim_{n \to \infty} \beta_i(n)$; see Corollary \ref{TWISTED}. It follows that the sequence $\{f_n(z)\}$ also converges as a sequence of formal power series, that is, for each $i \geq 0$ the coefficient $[z^i]f_n(z)$ is eventually constant in $n$. Its limit, the power series $g(z)$ with $[z^i]g(z) = \lim_{n \to \infty} [z^i] f_n(z)$, is
$$g(z) =  \lim^{z\text{-adic}}_{n \to \infty}   \sum_{i = 0}^{n^2} \frac{\beta_i(n) z^i}{(1-z^2)(1-z^4) \cdots (1-z^{2n})} =  \sum_{i =0}^{\infty} \frac{\beta_i z^i}{\prod_{r=1}^{\infty} (1-z^{2r})}.$$

To complete the proof, we wish to equate the pointwise limit $f(z)$ and the formal power series limit $g(z)$ of the sequence $\{ f_n(z)\}$. (These limits need not be equal in general, even when both limits exist and are analytic functions on the unit disk.) In order to do this, we will prove that the sequence $\{ f_n(z)\}$ converges {\it uniformly} to $f(z)$ on the closed disk $\left\{ z \in \mathbb{C} \; : \; |z| \leq \frac12 \right\}$, and then equality $f(z)=g(z)$ follows from an application of the Cauchy integral formula. 

To show uniform convergence, by Dini's theorem, it suffices to show that at each point $z$ in the closed disk, the sequence $\{f_n(z)\}$ is monotone increasing. But
 \begin{align*} & \sum_{n=0}^{\infty} \big( f_n(z) - f_{n-1}(z) \big) u^n = (1-u)  \sum_{n=0}^{\infty} f_n(z) u^n \\ 
& \qquad =  \frac{1}{v_{\mu}\, v_{\lambda}} \prod_{r=1}^{|\mu|} \left( \frac{u^r}{1-z^r} \right)^{n_r(\mu)}
\prod_{r=1}^{|\lambda|} \left( \frac{u^r}{1+z^r} \right)^{n_r(\lambda)} \prod_{r=1}^{\infty} \frac{1}{1-uz^{2r}},
\end{align*}
and by inspection, for each $n \geq0 $ and $z$ with $|z|\leq \frac12$, the coefficient $\big( f_n(z) - f_{n-1}(z) \big)$ of $u^n$ is positive as desired. 
We can therefore equate $f(z)=g(z)$, and we conclude
\[ \sum_{i =0}^{\infty} \beta_i z^i = \frac{1}{v_{\mu} \, v_{\lambda}} \prod_{r=1}^{|\mu|} \left( \frac{1}{1-z^r} \right)^{n_r(\mu)} \prod_{r=1}^{|\lambda|} \left( \frac{1}{1+z^r} \right)^{n_r(\lambda)}
 \]
as claimed.
\end{proof}

\begin{remark} Corollary \ref{BettiOGF} can also be proven using Theorem \ref{asym} and results of Jim\'enez Rolland and Wilson \cite{RW}. Specifically, combining Theorem \ref{asym}(ii) and Lehrer's identity Theorem \ref{classf} (ii) , we obtain the formula
\begin{align*} \lim_{n \rightarrow \infty}  \sum_{i= 0}^{n^2} q^{-i}  \beta_i(n) 
 =  \frac{1}{v_{\mu} v_{\lambda}} \prod_{r = 1}^{|\mu|}  \left( \frac{q^r}{q^r-1} \right)^{n_r(\mu)}
\prod_{r=1}^{|\lambda|}  \left( \frac{q^r}{q^r+1} \right)^{n_r(\lambda)}.
\end{align*}
By \cite[Theorem 4.3]{RW} the left-hand side of this formula is equal to $$ \sum_{i= 0}^{\infty} \lim_{n \rightarrow \infty} q^{-i} \beta_i(n) = \sum_{i= 0}^{\infty} q^{-i} \beta_i, $$ and so by the substitution $z=\frac1q$ we conclude Corollary \ref{BettiOGF}. 
The proof of Corollary \ref{BettiOGF} given above avoids the nontrivial combinatorial and analytic work performed in  \cite{RW}. 
\end{remark}
 
Using the generating function from Corollary \ref{BettiOGF}, we now prove Corollary \ref{RECURR}, on linear recurrence relations satisfied by the stable twisted Betti numbers.

\begin{proof}[Proof of Corollary \ref{RECURR}]

A general character polynomial $P$ may be written as a linear combination of character polynomials of the form ${X \choose \mu}{Y \choose \lambda}$ with $|\mu|+|\lambda| \leq \deg(P)$. Hence, its generating function can be written as a rational function with denominator given by the common denominator of the generating functions
$$  \frac{1}{v_{\mu} v_{\lambda}} \prod_{r=1}^{|\mu|} \left( \frac{1}{1-z^r} \right)^{n_r(\mu)} \prod_{r=1}^{|\lambda|} \left( \frac{1}{1+z^r} \right)^{n_r(\lambda)} \qquad |\mu|+|\lambda| \leq \deg(P).$$
A coarse bound on the degree of this denominator follows from the observation that, for each $1 \leq r \leq \deg(P)$, the corresponding factors $(1-z^r)$ and $(1+z^r)$ each appear in this common denominator with multiplicities at most $\frac{\deg(P)}{r}$, and hence together contribute at most $2r\left(\frac{\deg(P)}{r}\right)=2\deg(P)$ to its degree. This gives a total degree of at most $2\deg(P)^2$. \end{proof}

%{\color{cyan} Question for Jason and Rita: We can improve this bound on the common denominator by the observation that (1) in general, parts of size $r$ only really contribute   $2r\left \lfloor \frac{\deg(P)}{r}\right \rfloor \leq 2\deg(P)$ to the total degree, and (2) the polynomials $(1 \pm z^r)$ and $(1 \pm z^s)$ frequently share roots. I haven't worked out what the improved bound on $N$ would be. Do you think improving this bound warrants the extra combinatorics work? Response from Jason: I don't think this is worth the extra combinatorics work.}

The form of the generating functions in Corollary \ref{BettiOGF} also implies that the stable twisted Betti numbers $\beta_d$ will be quasipolynomial, as in Corollary \ref{quasipolynomiality}.

\begin{proof}[Proof of Corollary \ref{quasipolynomiality}] Stanley \cite[Proposition 4.4.1]{St1} proves that an integer function $p(t)$ is quasipolynomial with period $M$ if and only if the associated generating function $\sum_{n \geq 0}^{\infty} p(n) x^n$ is a rational function $A(x)/B(x)$ (reduced to lowest terms) such that all roots of $B(x)$ are $M^{th}$ roots of unity, and deg$(A)<$deg$(B)$. Moreover, the degree of $p(t)$ is strictly bounded above by the largest multiplicity of a root of $B(x)$. 

If the character polynomial $P$ has constant term zero, we may write $P$ as a linear combination of character polynomials ${X \choose \mu}{Y \choose \lambda}$ with $|\mu|+|\lambda|>0$. Then by Corollary 
 \ref{BettiOGF}, the generating function for the associated stable Betti numbers is given by a linear combination of functions of the form $$  \frac{1}{v_{\mu} v_{\lambda}} \prod_{i=1}^{|\mu|} \left( \frac{1}{1-z^i} \right)^{n_i(\mu)} \prod_{i=1}^{|\lambda|} \left( \frac{1}{1+z^i} \right)^{n_i(\lambda)} $$
such that $0<|\mu|+|\lambda| \leq \deg(P)$.
Since the denominators $(1-z^i)$ and $(1+z^i)$ have positive degree and all roots are $(2i)^{th}$ roots of unity, we conclude that these stable Betti numbers  are quasipolynomial, and have a quasiperiod bounded by the least common multiple of the numbers $2, 4, \ldots 2\,$deg$(P)$. Moreover, in the common denominator for this generating function the maximum multiplicity of any root is $\deg(P)$, and so the quasipolynomial has degree at most $\deg(P)-1$ as claimed.

Now suppose that $P$ has constant term $c$. Only the $0^{th}$ graded piece $R^0_n$ of the coinvariant algebra contains the trivial $B_n$ representation, hence the stable Betti numbers for $P$ and $(P-c)$ will be the same for all $i >0$. This concludes the proof.   \end{proof}

\subsection{Examples of stable twisted Betti numbers in type B/C} \label{SectionExamples}

We end the paper using Corollary \ref{BettiOGF} to compute some examples of stable Betti numbers.
Let $\Co^n$ denote the canonical $n$-dimensional $B_n$-representation by signed permutation matrices; this is the irreducible representation associated to the double partition $( (n-1), (1) )$.  Below we describe combinatorial properties of twisted Betti numbers with coefficients in $\Co^n$, and some symmetric and exterior powers of $\Co^n$ in small degree.

\begin{example}(Example: $\Co^n$).

Character polynomial:
\begin{align*}
P^{ \Co^n} = & X_1 - Y_1
\\ = & {X \choose \Y{1} } - {Y \choose \Y{1} }
\end{align*}

Betti numbers:
\begin{align*} \sum_{i=0}^{\infty}  \beta_i z^i  &=  \frac{z}{(1-z)(1+z)}  \\
= & z + z^3 + z^5 + z^7 + z^9 + \cdots + z^{2d+1} + \cdots \end{align*}

Recurrence: \qquad $\displaystyle \beta_d = \beta_{d-2} \qquad$ for $d \geq 3$ \\

Quasipolynomiality: \qquad For $d \geq 0$,  $\displaystyle \beta_d = \left\{ \begin{array}{cc} 0 & d \equiv 0 \pmod{2} \\  1 & d \equiv 1 \pmod{2} \end{array} \right. $ \\

\end{example}
\begin{example} ({$\bigwedge^2 \Co^n$}).

Character polynomial:
\begin{align*}
P^{\bigwedge^2 \Co^n} = & { X_1 \choose 2} + {Y_1 \choose 2} - X_2 +Y_2 -X_1Y_1
\\ = & {X \choose \Y{1,1} } + {Y \choose \Y{1,1} } -  {X \choose \Y{2} } + {Y \choose \Y{2} } - {X \choose \Y{1} } { Y \choose \Y{1} }
\end{align*}

Betti numbers:
\begin{align*} \sum_{i=0}^{\infty}  \beta_i z^i = & \frac{z^4}{(1-z^2)^2(1+z^2)}  \\
= &  z^4 +z^6 +2z^8 + 2z^{10} + 3z^{12} + 3z^{14} + \cdots + \left\lfloor \frac{d}{2} \right\rfloor z^{2d} + \cdots
\end{align*}

Recurrence: \qquad $\displaystyle \beta_d = \beta_{d-2} +  \beta_{d-4}  - \beta_{d-6}  \qquad$ for $d \geq 6$ \\

Quasipolynomiality: \qquad For $d \geq 0$,  $$\beta_d = \left\{ \begin{array}{cl} \frac{d}{4} & d \equiv 0 \pmod{4} \\[5pt]  \frac{d-2}{4} & d \equiv 2 \pmod{4}\\[5pt] 0 & d \equiv 1,3 \pmod{4}  \end{array} \right. $$

\end{example}

\begin{example}({$\bigwedge^3 \Co^n$}).

Character polynomial:
\begin{align*}
& { X_1 \choose 3} - { Y_1 \choose 3}  + X_1{ Y_1 \choose 2} - Y_1{ X_1 \choose 2} \\
& - X_1X_2 +X_2Y_1 + Y_2 X_1 -Y_1Y_2 + X_3 - Y_3
\end{align*}

% 1/(6 (1-z^3))-1/(6 (z^3+1))+1/(8 (z+1) (1-z^2))-1/(8 (z+1) (z^2+1))-1/(8 (1-z^2) (1-z))+1/(8 (z^2+1) (1-z))-1/(48 (z+1)^3)+1/(16 (z+1)^2 (1-z))-1/(16 (z+1) (1-z)^2)+1/(48 (1-z)^3)

Betti numbers:
\begin{align*}  \sum_{i=0}^{\infty} \beta_i z^i  =&  \;  \frac{z^9}{(1-z)^2(1+z)^2(1+z^2)(1-z^3)(1+z^3)}   \\
= & \; z^9+z^{11}+2 z^{13}+3 z^{15}+4 z^{17}+5 z^{19}+7 z^{21}+8 z^{23} \\
&+10 z^{25}+12 z^{27}+14 z^{29}+16 z^{31}+19 z^{33}+21 z^{35} \\
& +24 z^{37}+27 z^{39}+30 z^{41}+33 z^{43}+\cdots
\end{align*}

Recurrence: \qquad $$ \beta_d = \beta_{d-2} +  \beta_{d-4}  - \beta_{d-8} - \beta_{d-10} + \beta_{d-12}    \qquad \text{for } d \geq 12$$

Quasipolynomiality: \qquad For $d \geq 0$,  $$\displaystyle \beta_d = \left\{ \begin{array}{cl}
\frac{d^2}{48} - \frac{d}{8} + \frac{5}{48} & d \equiv 1, 5 \pmod{12} \\[5pt]
\frac{d^2}{48} - \frac{d}{8} + \frac{9}{48} & d \equiv 3 \pmod{12} \\[5pt]
\frac{d^2}{48} - \frac{d}{8} - \frac{7}{48} & d \equiv 7, 11 \pmod{12} \\[5pt]
\frac{d^2}{48} - \frac{d}{8} + \frac{21}{48} & d \equiv 9 \pmod{12} \\[5pt]
0 & d \equiv 0, 2, 4, 6, 8, 10 \pmod{12}  \end{array} \right. $$

\end{example}

\section*{Acknowledgements}
We are very grateful to Weiyan Chen for his generous answers to our questions and for suggesting directions to pursue with this approach. We also thank Robert Guralnick  and Jon Chaika for helpful discussions. We are also grateful to our referee for feedback, and in particular for catching a mistake in Corollary 1.17.

%{\small
%\bibliographystyle{amsalpha}
%\bibliography{MaxTori}\bigskip
%\qquad \\
%}

\end{document}